\documentclass{amsart}
\usepackage{amsmath,amsthm,amssymb,amsfonts}
\usepackage[colorlinks=true]{hyperref}

\hypersetup{urlcolor=blue, citecolor=blue}

\def\doi#1{   {\href{http://dx.doi.org/#1}
   {{\mdseries\ttfamily DOI}}}}

\usepackage{graphics}

\def\TT{{\mathbb T}}

\renewcommand{\phi}{\varphi}

\newcommand{\beeq}{\begin{equation}}\newcommand{\eneq}{\end{equation}}

\newenvironment{prf}{\noindent {\bf Proof.} }{\endprf\par}
\def \endprf{\hfill  {\vrule height6pt width6pt depth0pt}\medskip}

\def\<{\langle}             \def\>{\rangle}
\def\({\left(}                 \def\){\right)}
\numberwithin{equation}{section}
\newtheorem{thm}{Theorem}[section]
 
 \newtheorem{lem}[thm]{Lemma}

 \newtheorem{remark}[thm]{Remark}

\title[Positivity of temperature for some non-isothermal models]
{Positivity of temperature for some non-isothermal fluid models}

\author{Ning-An Lai}
\address{Institute of Nonlinear Analysis and Department of Mathematics\\
		Lishui University\\Lishui 323000, China}
		\email{ninganlai@lsu.edu.cn}

\author{Chun Liu}
\address{Department of Applied Mathematics, Illinois Institute of Technology, Chicago, IL 60616, USA}
\email{cliu124@iit.edu}

\author{Andrei Tarfulea}
\address{Department of Mathematics, Louisiana State University, Baton Rouge, LA 70803, USA}\email{tarfulea@lsu.edu}

\date{\today}

\begin{document}
\maketitle
\begin{abstract}
We establish three partial differential equation models describing the thermodynamic behavior of a fluid by combining the energetic variational approach, appropriate constitutive relations, and classical thermodynamic laws. By using a clear algebraic approach, we show a maximum/minimum principle for some auxiliary variables involving the absolute temperature $\theta$ and density $\rho$ under some special conditions, which then yields the positivity of the temperature. This important fact implies the thermodynamic consistency for our models.

\end{abstract}
\tableofcontents

\section{Introduction}
\noindent
The study of heat transfer in fluid dynamics has recently attracted growing attention. In non-isothermal models, the temperature is a non-constant material property that flows with the fluid. This creates local changes in the density and viscosity, which then cause the fluid flow to deviate from the classically expected outcome and further influence the rate of heat transfer. This two-way coupling phenomenon is prevalent in heat exchangers, chemical reactors, atmospheric flows, and processes in which components are cooled. One example we face almost every day is the noble gas flow in fluorescent lights. If thermal effects become significant in the flow of fluid through porous media, further applications appear: solidification of binary mixtures, dehumidification,
insulation, and heat pipes (see \cite{Kav} for more details). Besides the application in mechanical engineering, we also refer to \cite{BSKL} for an enhanced gas recovery application by non-isothermal compressible gas flow.\\
\\
Due to the prevalence of thermal effects in fluid flow, the literature has seen more and more non-isothermal fluid dynamic models established and delicately analyzed. We list a few of them here. By choosing an appropriate internal energy and entropy production, M$\acute{\text{a}}$lek and Pr$\mathring{\text{u}}\check{\text{s}}$a \cite{MP} arrive at a modern phenomenological theory (based in thermodynamics) of constitutive relations for compressible and incompressible viscous heat-conducting fluids (Navier-Stokes-Fourier), Korteweg fluids, and (in)compressible heat-conducting viscoelastic fluids (Oldroyd-B and Maxwell). The existence of weak solutions for the Navier-Stokes-Fourier system (see \cite{Gal, Zey}) describing the evolution of a Newtonian heat-conducting fluid in a bounded domain is systematically studied in \cite{Fei, NoPe}. Generalized non-isothermal compressible and incompressible non-Newtonian fluid systems are derived in \cite{KS} using the energetic variational approach (see \cite{Gya, HKL}),  which is based on Strutt \cite{Str} and Onsager \cite{Ons1, Ons2}. This approach combines the two systems derived from the least action principle and the maximum dissipation principle, and has been widely used to derive other non-isothermal models  by combining the basic thermodynamic laws.
We refer to \cite{LWL} for non-isothermal electrokinetics, \cite{DL} for the non-isothermal general Ericksen-Leslie system, \cite{HLLL} for the non-isothermal Poisson-Nernst-Planck-Fourier system, and \cite{SL} for the Brinkman-Fourier system with ideal gas equilibrium.\\
\\
It is believed that the positivity of the absolute temperature is a key postulate of
thermodynamics. Unfortunately, the crucial proof of this property for all times (starting from positive initial absolute temperature) is far from trivial in the non-isothermal setting. In \cite{DH}, the invalidity of negative temperature has been proved by demonstrating that it arises from the use of an entropy definition
that is inconsistent both mathematically and thermodynamically. In the recent series works on the Navier-Stokes-Fourier system by Feireisl et al., the definition of weak solution includes the restriction on absolute temperature $\theta$
\[
\theta>0,~~a.a.~in~(0, T)\times \Omega,
\]
and can actually prove the positivity of temperature if it emanates from positive initial data. See \cite{FS2016} for conditional regularity of very weak solutions to the Navier-Stokes-Fourier system, \cite{Fei2015} for the existence and stability of the weak solutions to the
Navier-Stokes-Fourier system and their relevance in the study of convergence of numerical schemes, and \cite{CFJP2020} for the existence of weak solutions to the stationary Navier-Stokes-Fourier system. In \cite{Tar}, the third author studied two hydrodynamic model problems (one incompressible and one compressible) with three-dimensional fluid flow on the torus and
temperature-dependent viscosity and conductivity, and established a positive lower bound for the temperature in each case. We also refer to the monograph \cite{Fei2009}, a detailed introduction to the singular limits and
scale analysis for the Navier-Stokes-Fourier system, in which the positivity of absolute temperature was proved by using a Poincar\'{e}-type inequality. We have to mention that thermodynamic consistency (positivity of absolute temperature) for a class of phase change models proposed by Fr\'{e}mond \cite{Fre} has been widely studied in the past, see \cite{CS, LSS, LuS, LuSS, SS} and references therein.\\
\\
This work focuses on the transport of heat in addition to fluid flow through porous media, and gives a general framework for deriving non-isothermal models by combining the energetic variational approach and some basic thermodynamic laws. From three different given free energies, we establish three different non-isothermal models, which we call \emph{non-isothermal ideal gas}, \emph{non-isothermal porous media}, and \emph{generalized non-isothermal porous media}. Moreover, in some special cases for the former two models, we find maximum/minimum principles for some auxiliary quantities related to the density and temperature, by adapting an idea originally from the work \cite{Tar}. This then implies the positivity of the absolute temperature. To avoid the problems connected with the boundary behavior of the fluid, we consider the problems on the torus and impose periodic boundary conditions.

\subsection{Outline}
In Section 2, we present a general derivation of models for non-isothermal fluid flow starting from a given free energy function. In Section 3,
we look at three free energy functions to obtain equations for the ideal gas, porous media, and generalized porous media equations.
In Section 4, we prove a priori maximum and/or minimum principles for the temperature and density of the ideal gas model, by first proving
them for certain auxiliary variables adapted to the structure of the equation. In Section 5, we prove analogous (but slightly weaker) results
for the generalized porous media equations.

\section{General Frame to Derive the Non-isothermal Model}
\noindent
We start from the free energy, $\Psi(\rho, \theta)$, which is a function of the density $\rho$ and the absolute temperature $\theta$.
With this in hand, we then define
\begin{itemize}
  \item (Specific) entropy of the system:
  \[
  \eta(\rho, \theta):=-\partial_{\theta}\Psi.\\
  \]
  \item (Specific) internal energy (the Legendre transform in $\theta$ of the free energy):
  \[
  e(\rho, \theta):=\Psi-\partial_\theta\Psi=\Psi+\eta\theta.\\
  \]
  \item Pressure:
  \[
  p=\Psi_{\rho}\rho-\Psi.\\
  \]
\end{itemize}
We also take $\rho$ and $\eta$ as new state variables and rewrite the internal energy function as
\[
e_1(\rho, \eta)=e\left(\rho, \theta(\rho, \eta)\right).
\]
A calculation shows that
\begin{equation}\label{deentropy}
\begin{aligned}
e_{1\eta}=\theta,~~~e_{1\rho}=\Psi_{\rho}.
\end{aligned}
\end{equation}
For the pressure we see that
\begin{equation}\label{ptheta}
\begin{aligned}
\partial_{\theta}p=\Psi_{\rho\theta}\rho-\Psi_{\theta}=\eta-\eta_{\rho}\rho.
\end{aligned}
\end{equation}
Additionally, a direct computation shows
\begin{equation}\label{nablap}
\begin{aligned}
\nabla p&=\nabla(\Psi_{\rho}\rho-\Psi)\\
&=(\Psi_{\rho\rho}\nabla\rho+\Psi_{\rho\theta}\nabla\theta)\rho+\Psi_{\rho}\nabla\rho-\Psi_{\rho}\nabla\rho-\Psi_{\theta}\nabla\theta\\
&=(\Psi_{\rho\rho}\nabla\rho+\Psi_{\rho\theta}\nabla\theta)\rho+\eta\nabla\theta\\
&=\rho\nabla e_{1\rho}+\eta\nabla e_{1\eta}.
\end{aligned}
\end{equation}
\noindent
We are now in position to derive a Darcy type diffusion law.
\begin{lem}
Let $u(t, x)$ be the velocity field of the fluid. Then, 
\begin{equation}\label{darcy}
\begin{aligned}
\nabla p=-\rho u.
\end{aligned}
\end{equation}
\end{lem}
\begin{proof}

As in \cite{GKL, HKL}, the total energy and dissipation are defined by
\[
E^\text{total}=\int_{\Omega_t^x}\Psi(\rho, \theta)dx, ~~~\mathcal{D}^\text{total}=\frac12\int_{\Omega_t^x}\rho u^2dx,
\]
where $\Omega_t^x$ is the deformed configuration corresponding to the reference configuration $\Omega_0^X$; here $X$ represents
Lagrangian coordinates and $x$ represents Eulerian coordinates. We first rewrite the energy functional in the Lagrangian coordinate system, and then take the least action principle conservative force
\begin{equation}\label{Lag}
\begin{aligned}
A(x(X, t))=-\int_0^T\int_{\Omega_0^X}\Psi\left(\frac{\rho_0(X)}{J}, \theta_0(X)\right)JdXdt,
\end{aligned}
\end{equation}
where $J=\det F$ and $F=\frac{\partial x}{\partial X}$ denotes the deformation gradient. Then, taking the variation for any smooth compactly supported $y(X, t)=\widetilde{y}(x(X, t), t)$ with respect to $x$ yields
\begin{equation}
 \begin{aligned}
 &\frac{d}{d\varepsilon}\bigg|_{\varepsilon=0}A(x(X,t)+\varepsilon y(X,t))=\frac{d}{d\varepsilon}\bigg|_{\varepsilon=0}A(x+\varepsilon y)\\
 &=-\frac{d}{d\varepsilon}\bigg|_{\varepsilon=0}\int_{0}^{T}\int_{\Omega^{X}_{0}}
 \Psi\bigg(\frac{\rho_0(X)}{{\rm det}\frac{\partial(x+\varepsilon y)}{\partial X}}, \theta_0(X)\bigg)\bigg({\rm det}\frac{\partial(x+\varepsilon y)}{\partial X}\bigg)dXdt\\
 &=-\int_{0}^{T}\int_{\Omega^{X}_{0}}\bigg(\frac{d}{d\varepsilon}
 \bigg|_{\varepsilon=0}\Psi\bigg(\frac{\rho(X)}{{\rm det}\frac{\partial(x+\varepsilon y)}{\partial X}}, \theta_0(X)\bigg)\bigg)
 \cdot\bigg({\rm det}\frac{\partial x}{\partial X}\bigg)\\
 &-\Psi\bigg(\frac{\rho_{0}(X)}{\rm det\frac{\partial x}{\partial X}}, \theta_0(X)\bigg)\cdot \frac{d}{d\varepsilon}\bigg|_{\varepsilon=0}\bigg({\rm det}\frac{\partial(x+\varepsilon y)}{\partial X}\bigg)dXdt\\
 &=-\int_{0}^{T}\int_{\Omega^{X}_{0}}
 \Psi_{\rho}\bigg(\frac{\rho_{0}(X)}{J}, \theta_0(X)\bigg)\cdot \bigg(-\frac{\rho_{0}(X)}{J^{2}}\bigg)\cdot J\cdot {\rm tr}\bigg(\frac{\partial X}{\partial x}\frac{\partial y}{\partial X}\bigg)\cdot J\\
 &+\Psi\bigg(\frac{\rho_{0}(X)}{J},\theta_0(X)\bigg)\cdot J\cdot {\rm tr}\bigg(\frac{\partial X}{\partial x}\frac{\partial y}{\partial X}\bigg)dXdt\\
 &=\int_{0}^{T}\int_{\Omega^{x}_{t}}\left(\Psi_{\rho}(\rho(x,t))\cdot \rho(x,t)-\Psi(\rho(x,t))\right)\cdot(\nabla_{x}\cdot \widetilde{y})dxdt\\
 &=\int_{0}^{T}\int_{\Omega^{x}_{t}}-\nabla_{x}\left(\Psi_{\rho}(\rho(x,t))\cdot \rho(x,t)-\Psi(\rho(x,t))\right)\cdot \widetilde{y}dxdt,\\
 \end{aligned}
\end{equation}
which gives the conservative force by combining with the definition of pressure $p$
\begin{equation}\label{conforce}
\begin{aligned}
\text{Force}_{\text{cons}}&=-\nabla_{x}\left(\Psi_{\rho}(\rho(x,t))\cdot \rho(x,t)-\Psi(\rho(x,t))\right)\\
&=-\nabla p.\\
\end{aligned}
\end{equation}
On the other hand, according to the maximum dissipation principle, taking variation with respect to the velocity to the dissipation functional, we may get the dissipative force
\begin{equation}\label{disforce}
\begin{aligned}
\delta_u\mathcal{D}=\int_{\Omega_t^x}\rho u\delta udx=\int_{\Omega_t^x}\text{Force}_\text{diss}\delta udx.\\
\end{aligned}
\end{equation}
With these two forces in hand, we may then apply Newton's force balance law, which states that all forces, both conservative and dissipative in kind, add up to zero (``action'' equals ``reaction''), thus
\[
\text{Force}_\text{cons}= \text{Force}_\text{diss},
\]
which yields \eqref{darcy}, by combining \eqref{conforce} and \eqref{disforce}.
\end{proof}
\noindent
Next we introduce some classical thermodynamic laws.
\begin{itemize}
\item The first one is related to the rate of change of the internal energy with dissipation and heat:
\begin{equation}\label{firstlaw}
\begin{aligned}
\frac{de}{dt}=\nabla\cdot W+\nabla\cdot q.
\end{aligned}
\end{equation}
Here, $W$ denotes the amount of thermodynamic work done by the system on its surroundings and $q$ denotes the quantity of energy supplied to the system as heat.
\item The second one is related to the change of entropy
\begin{equation}\label{secondlaw}
\begin{aligned}
\partial_t\eta+\nabla\cdot(\eta u)=\nabla\cdot \left(\frac{q}{\theta}\right)+\Delta,
\end{aligned}
\end{equation}
where $\Delta\ge 0$ denotes the rate of entropy production.
\item The third one is called Fourier's law
\begin{equation}\label{Flaw}
\begin{aligned}
q=\kappa_3\nabla \theta,
\end{aligned}
\end{equation}
where $\kappa_3$ denotes the material conductivity which may depend on $\rho$ and $\theta$.
\end{itemize}
Also we make the general kinematic assumption of mass transport
\begin{equation}\label{mass}
\begin{aligned}
\rho_t+\nabla\cdot(\rho u)=0.
\end{aligned}
\end{equation}
By \eqref{nablap}, \eqref{firstlaw}, \eqref{secondlaw}, \eqref{Flaw} and \eqref{mass}, one has
\begin{equation}\label{pre1}
\begin{aligned}
&\frac{de_1(\rho, \eta)}{dt}\\
=&e_{1\rho}\rho_t+e_{1\eta}\eta_t\\
=&e_{1\rho}\left(-\nabla\cdot(\rho u)\right)+e_{1\eta}\left(-\nabla\cdot(\eta u)+\nabla\cdot\left(\frac{q}{\theta}\right)+\Delta\right)\\
=&-\nabla\cdot\left(e_{1\rho}\rho u+e_{1\eta}\eta u\right)+\left(\rho\nabla e_{1\rho}+\eta\nabla e_{1\eta}\right)\cdot u+\theta\nabla\cdot\left(\frac{q}{\theta}\right)+\theta\Delta\\
=&\nabla\cdot W+\nabla p\cdot u+\nabla\cdot q-\frac{q}{\theta}\cdot\nabla\theta+\theta\Delta\\
=&\nabla\cdot W-\rho u^2+\nabla\cdot q-\frac{\kappa_3|\nabla\theta|^2}{\theta}+\theta\Delta,\\
\end{aligned}
\end{equation}
where
\[
W=-\left(e_{1\rho}\rho+e_{1\eta}\eta\right)u
\]
denotes the work done by the system. Set the rate of entropy production by
\begin{equation}\label{delta}
\begin{aligned}
\Delta&=\frac{1}{\theta}\left(\rho|u|^2+\frac{q\cdot\nabla\theta}{\theta}\right)\\
&=\frac{1}{\theta}\left(\rho|u|^2+\frac{\kappa_3|\nabla\theta|^2}{\theta}\right),\\
\end{aligned}
\end{equation}
then \eqref{pre1} turns out to be \eqref{firstlaw}.\\
\\
With the above preliminaries in hand, we may establish the non-isothermal equation by \eqref{secondlaw}. By combining \eqref{ptheta} and \eqref{mass}, we have
\begin{equation}\label{pre2}
\begin{aligned}
&\eta_t+\nabla\cdot(\eta u)\\
=&\eta_\theta(\theta_t+u\cdot\nabla\theta)+\eta_{\rho}(\rho_t+u\cdot\nabla\rho)+\eta\nabla\cdot u\\
=&\eta_\theta(\theta_t+u\cdot\nabla\theta)+\eta_{\rho}\left(-\rho\nabla\cdot u\right)+\eta\nabla\cdot u\\
=&\eta_\theta(\theta_t+u\cdot\nabla\theta)+\left(\eta-\eta_{\rho}\rho\right)\nabla\cdot u\\
=&\eta_\theta(\theta_t+u\cdot\nabla\theta)+\partial_{\theta}p\nabla\cdot u\\
=&\nabla\cdot j+\Delta\\
=&\nabla\cdot\left(\frac{q}{\theta}\right)+\frac{1}{\theta}\left(\rho|u|^2+\frac{q\cdot\nabla\theta}{\theta}\right),
\end{aligned}
\end{equation}
which yields
\begin{equation}\label{pre3}
\begin{aligned}
&\eta_\theta(\theta_t+u\cdot\nabla\theta)+\partial_{\theta}p\nabla\cdot u\\
=&\nabla\cdot\left(\frac{q}{\theta}\right)+\frac{1}{\theta}\left(\rho|u|^2+\frac{q\cdot\nabla\theta}{\theta}\right).
\end{aligned}
\end{equation}
Hence, for a given free energy $\Psi(\rho, \theta)$, the corresponding non-isothermal model is obtained by combining \eqref{darcy}, \eqref{mass} and \eqref{pre3}.

\section{Thermodynamics Models}
\noindent
In this section we introduce three specific free energies, which are respectively related to non-isothermal ideal gas, non-isothermal porous media, and non-isothermal generalized porous media. We then derive their corresponding equations.

\subsection{Ideal Gas}

For the ideal gas, the free energy is given by
\[
\Psi(\rho, \theta)=\kappa_1\theta\rho\ln\rho-\kappa_2\rho\theta\ln\theta.
\]
Then, we have
\begin{equation}\label{idp}
\begin{aligned}
&p=\kappa_1\rho\theta,\\
&\partial_{\theta}p=\kappa_1\rho,\\
&\eta_{\theta}=\frac{\kappa_2\rho}{\theta},\\
&\rho u=-\nabla p=-\kappa_1\nabla(\rho\theta).\\
\end{aligned}
\end{equation}
Thus, the mass equation \eqref{mass} becomes
\begin{equation}\label{idmass}
\begin{aligned}
\partial_t\rho&=-\nabla\cdot(\rho u)\\
&=\kappa_1\nabla\cdot\left(\nabla(\rho\theta)\right)\\
&=\kappa_1\Delta(\rho\theta).\\
\end{aligned}
\end{equation}
Therefore, \eqref{pre3} changes to
\begin{equation}\label{idpre3}
\begin{aligned}
&\frac{\kappa_2\rho}{\theta}(\theta_t+u\cdot\nabla\theta)+\kappa_1\rho\nabla\cdot u\\
=&\nabla\cdot\left(\frac{\kappa_3\nabla\theta}{\theta}\right)+\frac{1}{\theta}
\left(-\kappa_1\nabla(\rho \theta)\cdot u+\frac{\kappa_3|\nabla\theta|^2}{\theta}\right),\\
\end{aligned}
\end{equation}
which implies
\begin{equation}\label{idpre4}
\begin{aligned}
\kappa_2(\rho\theta)_t-\kappa_2\theta\rho_t+\kappa_2\rho u\cdot\nabla\theta+\kappa_1\nabla\cdot(\rho\theta u)=\theta\nabla\cdot\left(\frac{\kappa_3\nabla\theta}{\theta}\right)
+\frac{\kappa_3|\nabla\theta|^2}{\theta},\\
\end{aligned}
\end{equation}
which can be simplified to
\begin{equation}\label{idpre5}
\begin{aligned}
\kappa_2(\rho\theta)_t-\kappa_1(\kappa_1+\kappa_2)\nabla
\cdot\left(\theta\nabla(\rho\theta)\right)=\nabla\cdot\left(\kappa_3\nabla\theta\right).
\end{aligned}
\end{equation}
Hence, we get the non-isothermal model for ideal gas:
\begin{equation}\label{idmodel}
\left \{
\begin{aligned}
&\partial_t\rho=\kappa_1\Delta(\rho\theta),\\
&\kappa_2(\rho\theta)_t-\kappa_1(\kappa_1+\kappa_2)\nabla\cdot
\left(\theta\nabla(\rho\theta)\right)=\nabla\cdot\left(\kappa_3\nabla\theta\right).\\
\end{aligned} \right.
\end{equation}

\subsection{Porous Media}

For the porous media, we introduce the free energy as
\begin{equation}\label{pmfreeenergy}
\begin{aligned}
\psi(\rho, \theta):=\kappa_1\theta\rho^{\alpha}-\kappa_2\rho\theta\ln\theta,
\end{aligned}
\end{equation}
with $\alpha > 1$. In this case we have
\begin{equation}\label{pmp}
\begin{aligned}
&p=\kappa_1(\alpha-1)\theta\rho^{\alpha},\\
&\partial_{\theta}p=\kappa_1(\alpha-1)\rho^{\alpha},\\
&\eta_{\theta}=\frac{\kappa_2\rho}{\theta},\\
&\rho u=-\nabla p=-\kappa_1(\alpha-1)\nabla(\theta\rho^{\alpha})\\
\end{aligned}
\end{equation}
and
\begin{equation}\label{pmmass}
\begin{aligned}
\partial_t\rho&=-\nabla\cdot(\rho u)\\
&=\kappa_1(\alpha-1)\nabla\cdot\left(\nabla(\theta\rho^{\alpha})\right)\\
&=\kappa_1(\alpha-1)\Delta(\theta\rho^{\alpha}).\\
\end{aligned}
\end{equation}
Also, \eqref{pre3} becomes
\begin{equation}\label{pmpre3}
\begin{aligned}
&\frac{\kappa_2\rho}{\theta}(\theta_t+u\cdot\nabla\theta)
+\kappa_1(\alpha-1)\theta\rho^{\alpha-1}\nabla\cdot u\\
=&\nabla\cdot\left(\frac{\kappa_3\nabla\theta}{\theta}\right)+\frac{1}{\theta}
\left(-\kappa_1(\alpha-1)\nabla(\theta\rho^{\alpha})\cdot u+\frac{\kappa_3|\nabla\theta|^2}{\theta}\right),\\
\end{aligned}
\end{equation}
which yields finally
\begin{equation}\label{pmpre4}
\begin{aligned}
\kappa_2(\rho\theta)_t+\kappa_2\nabla\cdot(\rho\theta u)+\kappa_1(\alpha-1)\nabla\cdot(\theta\rho^{\alpha}u)=\nabla\cdot(\kappa_3\nabla\theta),
\end{aligned}
\end{equation}
and hence we get the non-isothermal porous media system:
\begin{equation}
\label{nonisopm}
\left\{
\begin{aligned}
&\partial_t\rho=\kappa_1(\alpha-1)\Delta\left(\theta\rho^{\alpha}\right),\\
&\kappa_2(\rho\theta)_t-\kappa_1\kappa_2(\alpha-1)\nabla\cdot
\left(\theta\nabla(\theta\rho^{\alpha}) \right)-\kappa_1^2(\alpha-1)^2\nabla\cdot
\left(\theta\rho^{\alpha-1}\nabla(\theta\rho^{\alpha})\right)\\
&=\nabla\cdot(\kappa_3\nabla\theta).\\
\end{aligned}
\right.
\end{equation}

\subsection{Generalized Porous Media}

We introduce the free energy for generalized porous media as
\[
\Psi(\rho, \theta)=k_1\theta\rho^{\alpha}-k_2\rho\theta^{\beta},
\]
with $\alpha, \beta > 1$. Then
\begin{equation}\label{gpmp}
\begin{aligned}
&p=k_1(\alpha-1)\theta\rho^{\alpha},\\
&\partial_{\theta}p=k_1(\alpha-1)\rho^{\alpha},\\
&\eta_{\theta}=k_2\beta(\beta-1)\rho\theta^{\beta-2},\\
&\rho u=-\nabla p=-k_1(\alpha-1)\nabla(\theta\rho^{\alpha})\\
\end{aligned}
\end{equation}
and
\begin{equation}\label{gpmmass}
\begin{aligned}
\partial_t\rho&=-\nabla\cdot(\rho u)\\
&=k_1(\alpha-1)\nabla\cdot\left(\nabla(\theta\rho^{\alpha})\right)\\
&=k_1(\alpha-1)\Delta(\theta\rho^{\alpha}).\\
\end{aligned}
\end{equation}
In this case \eqref{pre3} becomes
\begin{equation}\label{gpmpre3}
\begin{aligned}
&k_2\beta(\beta-1)\rho\theta^{\beta-2}(\theta_t+u\cdot\nabla\theta)+k_1(\alpha-1)\rho^{\alpha}\nabla\cdot u\\
=&\nabla\cdot\left(\frac{k_3\nabla\theta}{\theta}\right)+\frac{1}{\theta}\left(-k_1(\alpha-1)\nabla(\theta\rho^{\alpha})\cdot u+\frac{k_3|\nabla\theta|^2}{\theta}\right),\\
\end{aligned}
\end{equation}
which yields finally
\begin{equation}\label{pmpre4}
\begin{aligned}
&k_2(\beta-1)(\rho\theta^{\beta})_t-k_1k_2(\alpha-1)(\beta-1)\nabla\cdot\left(\theta^{\beta}\nabla(\theta\rho^{\alpha}) \right)\\
&-k_1^2(\alpha-1)^2\nabla\cdot\left(\theta\rho^{\alpha-1}\nabla(\theta\rho^{\alpha})\right)=\nabla\cdot(k_3\nabla\theta),
\end{aligned}
\end{equation}
and hence we obtain the non-isothermal porous media system:
\begin{equation}
\label{gennonisopm}
\left\{
\begin{aligned}
&\partial_t\rho=k_1(\alpha-1)\Delta\left(\theta\rho^{\alpha}\right),\\
&k_2(\beta-1)(\rho\theta^{\beta})_t-k_1k_2(\alpha-1)(\beta-1)\nabla\cdot\left(\theta^{\beta}\nabla(\theta\rho^{\alpha}) \right)\\
&-k_1^2(\alpha-1)^2\nabla\cdot\left(\theta\rho^{\alpha-1}\nabla(\theta\rho^{\alpha})\right)=\nabla\cdot(k_3\nabla\theta).\\
\end{aligned}
\right.
\end{equation}

\section{The Maximum/Minimum Principle for Thermal Ideal Gas Model}
\noindent
In this section, we use the structure of \eqref{idmodel} to establish maximum and/or minimum principles for
certain auxiliary variables in the temperature and pressure. Even if one assumes a priori that a smooth solution
pair $(\rho, \theta)$ exists, it is not feasible to obtain max/min principles for the two functions directly due to the
complicated interdependence between $\rho$ and $\theta$. Indeed, maximum principles for coupled systems of
partial differential equations are notoriously hard to obtain, and is one of the major obstacles in going from
``scalar-valued'' problems (e.g., heat, porous media, or surface quasigeostrophic equation) to
``vector-valued'' problems (e.g., Navier-Stokes and Euler equations). One might then search for a ``state variable''
$\lambda(\rho,\theta)$ that is (super- or sub-) conserved, but identifying the right variable $\lambda$ is nontrivial.\\
\\
Nevertheless, if the material conductivity $\kappa_3$ is proportional to $\theta \rho$, we can find two homogeneous
auxiliary variables $\theta \rho^{1+\gamma_\pm}$ that, through careful cancellation in the structure of
\eqref{idmodel}, satisfy pointwise a priori maximum or minimum principles, in the form of Theorem \ref{thm:idmodel-maxmin}
below. Rather than simply verifying the principle for the two auxiliary variables above, the proof takes a more general
approach. It will look at variables $\theta \rho f(\rho)$ for a to-be-determined positive weight function $f$ and,
using the structure of \eqref{idmodel}, show that $f$ must satisfy one of two possible ordinary differential equations,
which naturally lead to the variables above. The method is similar to how \cite{Tar} found appropriate temperature weights
to close energy-type estimates for a fluid equation with thermal dissipation. The proof below
is also slightly more general, as the two functions $f$ can
also be obtained implicitly when $\kappa_3 = D(\rho) \theta$, in terms of the function $D$. Note that the ansatz
$\theta \rho f(\rho)$ for the auxiliary variables essentially covers any ``state variable'' of the form
$\lambda(\rho,\theta) = \theta^{\gamma} \tilde{f}(\rho)$ (for $\gamma > 0$), so that the auxiliary variables found in the
proof below are the only ones in this class to satisfy maximum or minimum principles.

\begin{thm}\label{thm:idmodel-maxmin}
Consider the non-isothermal ideal gas model \eqref{idmodel} on $[0,T) \times \mathbb{T}^n$.
Assume we have a smooth solution pair $(\rho \theta)$ on this domain. If the material conductivity $\kappa_3$ takes the form
\[
\kappa_3=\kappa_1\widetilde{D}\theta\rho,
\]
for $\widetilde{D}> $ a fixed constant, then we have\\
\noindent\textbf{I.} The absolute temperature is positive on $\mathbb{T}^n\times [0, T)$.\\
\noindent\textbf{II.} The density is bounded from above unconditionally.\\
\noindent\textbf{III.} If the temperature $\theta(t, x)$ either blows up or goes to zero, then the density $\rho(t, x)$ must vanish. Precisely, we have
\begin{equation}\label{drate}
\begin{aligned}
\rho(t, x)\le \min\left\{\theta^{-c_1}, \theta^{-c_2}\right\}
\end{aligned}
\end{equation}
for some constants $c_1>0, c_2<0$ depending on $\kappa_1, \kappa_2, \tilde{D}$.
\end{thm}

\begin{proof}
We can slightly simplify this system by writing $\beta = 1+\kappa_1/\kappa_2$ and (by abuse of notation)
replacing $\kappa_3$ by $\kappa_3 \kappa_2$ to obtain
\begin{equation}
\left\lbrace
\begin{split}
&\partial_t \rho = \kappa_1 \Delta(\rho \theta) \\
&\partial_t (\rho \theta) = \nabla \cdot \left( \beta \kappa_1 \theta \nabla(\rho \theta)
+\kappa_3 \nabla \theta \right)
\end{split}
\right.
\label{eq:tig}
\end{equation}
Here $\kappa_3$ is not necessarily constant; it generally depends on $\theta$ and $\rho$. In the interest of
making the proof more general, we will initially take $\kappa_3 = \theta D(\rho)$. Later, we make the special assumption that
$D(\rho) = \kappa_1 \tilde{D} \rho$.\\
\\
Let $f: (0,\infty) \rightarrow (0,\infty)$ be a smooth monotone function to be determined later.
We look at the quantity $f(\rho) \rho \theta$. Taking a derivative and using \eqref{eq:tig} yields
\begin{equation}
\partial_t (f\rho\theta) = \theta f' \rho \partial_t \rho + f \partial_t (\rho \theta)
= \kappa_1 \rho \theta f' \Delta (\rho \theta) +
f \nabla \cdot \left( \beta\kappa_1 \theta \nabla (\rho \theta) + \kappa_3 \nabla\theta \right).
\label{eq:001}
\end{equation}
Let $x_0 \in \TT^n$ be a point where $f\rho\theta$ achieves a local minimum (respectively maximum) in space.
Assume that $f(\rho(x_0))\rho(x_0)\theta(x_0) > 0$.
For the rest of the proof, all quantities are implicitly evaluated at $x_0$, though we suppress the notation.
Then the gradient at this point vanishes, so that
\begin{equation}
(f' \rho + f) \theta \nabla \rho + f\rho\nabla \theta = 0,
\label{eq:crit-grad}
\end{equation}
and the quantity
\begin{align*}
L &:= \Delta(f\rho\theta) = (2f'+f''\rho) \theta |\nabla \rho|^2 + 2(f'\rho+f) \nabla\theta \cdot \nabla\rho
+(f'\rho+f)\theta\Delta\rho + f\rho\Delta\theta \\
&=\frac{|\nabla \theta|^2}{\theta} \left( \frac{(f \rho)^2(2f'+f''\rho)}{(f'\rho+f)^2}-2f\rho \right)
+ (f'\rho+f)\theta\Delta\rho+f\rho\Delta\theta
\end{align*}
is nonnegative (respectively nonpositive). The last equality used \eqref{eq:crit-grad}, which in general
allows us to compare terms of the form $|\nabla\rho|^2$, $\nabla\rho\cdot\nabla\theta$, and $|\nabla\theta|^2$
to each other (recall that they are all evaluated at $x_0$).\\
\\
Taking $\kappa_3 = \theta D(\rho)$ and expanding \eqref{eq:001} (again using \eqref{eq:crit-grad}) yields
\begin{align*}
\partial_t (f\rho\theta) =& \kappa_1 \theta f' \rho \Delta(\rho\theta)
+ f \Big( \beta\kappa_1 (\theta\nabla\theta\cdot\nabla\rho + \rho|\nabla\theta|^2 + \theta\Delta(\rho\theta))\\
&+D |\nabla\theta|^2 + D \theta \Delta\theta + D' \theta \nabla\theta\cdot\nabla\rho \Big) \\
=& \left( \kappa_1 \theta f'\rho + \beta \kappa_1 \theta f \right) \theta\Delta\rho
+ \left( \kappa_1 \theta f'\rho + \beta\kappa_1 \theta f + \theta f \frac D \rho \right) \rho\Delta\theta \\
&+ \left( -\frac{2\kappa_1 f f' \rho^2}{f'\rho+f}+\beta\kappa_1 \left( f\rho - \frac{3f^2\rho}{f'\rho+f} \right)
- \frac{f^2 \rho D'}{f'\rho+f} + f D \right) |\nabla \theta|^2 \\
=& \theta (F_1 \theta\Delta\rho + F_2 \rho\Delta\theta) + F_3 |\nabla \theta|^2.
\end{align*}
The goal then is to choose $f$ (in terms of $\beta$, $\kappa_2$, and $D$) such that we may
rewrite the above as
\begin{equation}
\partial_t (f\rho\theta) = \theta \tilde{F} L + \tilde{G} |\nabla \theta|^2,
\label{eq:002}
\end{equation}
with $\tilde{F} \geq 0$ and $\tilde{G}$ nonnegative (respectively nonpositive).
This would show that $f\rho\theta$
satisfies a minimum (respectively maximum) principle for such $f$.\\
\\
In order for \eqref{eq:002} to hold, there must be some $\lambda \geq 0$ (not necessarily constant)
such that
\begin{align*}
&\kappa_1 f'\rho + \beta\kappa_1 f = \lambda (f'\rho+f), \\
&\kappa_1 f'\rho + \beta\kappa_1 f + f \frac D \rho = \lambda f.
\end{align*}
Subtracting yields
$$ \lambda = -\frac{f D}{f'\rho^2}, $$
which immediately implies that
\begin{equation}
f' < 0.
\label{requirement1}
\end{equation}
We therefore need $f$ to satisfy
$$ \kappa_1 f'\rho + \beta \kappa_1 f + \frac{fD}{\rho} + \frac{f^2 D}{f'\rho^2} = 0. $$
Solving for $f'$ formally yields
\begin{equation}
f' = f \frac{-\beta \kappa_1 \rho - D \pm \sqrt{\beta^2\kappa_1^2 \rho^2
+2(\beta-2)\kappa_1 \rho D + D^2}}{2\kappa_1 \rho^2}
\label{eq:ODE}
\end{equation}
First, we remark that \eqref{requirement1} is always satisfied whenever the right-hand-side of \eqref{eq:ODE} is
real-valued. The numerator is of the form
$-b \pm \sqrt{b^2 - 4ac}$ where each of $a$, $b$, and $c$ are positive (for $\rho>0$).\\
\\
Second, as long as $\kappa_1$ and $\kappa_2$ are positive (so $\beta > 1$), the discriminant of \eqref{eq:ODE} is strictly
positive for all $\rho$ (regardless of the value of $D$). This is easily seen with the Schwartz inequality,
and that the
discriminant is bounded below by $(\beta^2-1)\kappa_2^2 \rho^2$. Thus, \eqref{eq:ODE} specifies two ODE's that
are locally well-posed for all $\rho>0$.\\
\\
Third, it is immediate from \eqref{eq:ODE} that $|f'| \leq C |f| (\rho^{-1} + D(\rho)\rho^{-2})$. Gronwall's
inequality then guarantees that both ODE's are globally well-posed on $(0,\infty)$. That is, given $\rho_0>0$
and any initial datum $f_0>0$, there exist two unique, positive, monotone-decreasing weight functions $f_\pm$
defined on $(0,\infty)$ such that $f_\pm$ satisfies the corresponding ODE of \eqref{eq:ODE} pointwise
and $f_\pm(\rho_0) = f_0$.\\
\\
Fourth, we can directly apply the Duhamel principle. Writing \eqref{eq:ODE} as
$$ f'(\rho) = f(\rho) \Gamma_\pm(\beta,\kappa_1,D,\rho), $$
we get that
$$ f_\pm(\rho) = f_0 \exp \left( \int_{\rho_0}^\rho \Gamma_\pm(\beta,\kappa_1,D(r),r)dr \right). $$
In the special case $D(\rho) = \tilde{D} \kappa_1 \rho$ (for $\tilde{D} > 0$ a fixed constant), an assumption
we will make for the remainder of
this proof, the functions are given explicitly as $f_\pm(\rho) = f_\pm(1) \rho^{\gamma_\pm}$
with
$$ \gamma_\pm = \frac{-\beta - \tilde{D} \pm
\sqrt{\beta^2 + 2(\beta-2) \tilde{D} + \tilde{D}^2}}
{2}. $$
Note that $\gamma_{\pm} < 0$ for all values of $\tilde{D}$, $\beta$, and $\kappa_1$.
Moreover, $1+\gamma_+ > 0$ while $1+\gamma_- < 0$.

Thus we do obtain \eqref{eq:002}, in the sense that
$$ \partial_t (f\rho\theta) = -\theta\frac{f D}{f'\rho^2} L + \tilde{G} |\nabla \theta|^2, $$
where
\begin{equation}
\begin{aligned}
\tilde{G} =& \frac{f D}{f'\rho^2} \left( \frac{(f\rho)^2(2f'+f''\rho)}{(f'\rho+f)^2}-2f\rho \right)\\
&+\left( fD + (\beta-2)\kappa_1 f\rho - (D' + \kappa_1 (3\beta-2)) \frac{f^2 \rho}{f'\rho+f} \right).
\label{eq:G}
\end{aligned}
\end{equation}

\begin{remark}
This is a complicated expression which can be reduced by repeated use of \eqref{eq:ODE} until it only
involves $f$, $\rho$, $D$, and $D'$. Notice that $\tilde{G}$ does not depend on $\theta$ or on the position
or time variables. Its sign dictates the nature of the max/min principle satisfied by the auxiliary variable
$\theta \rho f(\rho)$, and this ultimately depends only on $\rho$. If $D$ is left as a generic (monotone) function of $\rho$, the range of
possible behaviors is quite complicated, in some cases leading to ``banded'' structure where the auxiliary variable
satisfies a maximum principle in certain interval ranges of $\rho$ and a minimum principle on
the complementary intervals (and \emph{both} at the endpoints). For this ideal gas
model, one obtains much more precise and unconditional results if one adheres to the special case
$D = \tilde{D} \rho$, with more general functions $D$ left for future work.
\end{remark}

Since $f'\rho= \gamma_\pm f$, \eqref{eq:G} reduces to the simpler expression
\begin{equation}
\tilde{G}_\pm =  \kappa_1 f \rho \left( -\frac{(2+\gamma_\pm)\tilde{D}}{\gamma_\pm(1+\gamma_\pm)}
+ (\tilde{D}+\beta-2) \frac{\gamma_\pm}{1+\gamma_\pm} - \frac{2\beta}{1+\gamma_\pm} \right).
\end{equation}
Recall that the minimum principle requires that $\tilde{G} \geq 0$, but the maximum principle
requires that $\tilde{G} \leq 0$. However, we see that $\tilde{G}_\pm$ is in fact $C_\pm f \rho$ for some
\emph{fixed} constants $C_\pm$ that depend on the initial parameters. This guarantees that each of
$f_\pm \rho \theta$ individually satisfy \emph{either} a maximum principle \emph{or} a minimum principle
for the entire lifetime of the solution $(\rho,\theta)$.\\
\\
Observing that $\gamma_+ \gamma_- = \tilde{D}$, a calculation shows that
\begin{align*}
\tilde{G}_+ &= \frac{\kappa_1 f \rho}{1+\gamma_+} \left( -\tilde{D} - 2\gamma_- + \tilde{D} \gamma_+
+ \beta \gamma_+ - 2 \gamma_+ - 2 \beta \right) \\
&= \frac{\kappa_1 f \rho}{1+\gamma_+} \left( \tilde{D}(1+ \gamma_+) + \beta \gamma_+ \right).
\end{align*}
Recall that $1+\gamma_+ > 0$. We claim that $\tilde{D}(1+\gamma_+) + \beta \gamma_+$ is always negative.
This is true if and only if
$$ (\tilde{D}+\beta) \sqrt{(\tilde{D}+\beta)^2 - 4\tilde{D}} \leq \tilde{D}^2+2\beta\tilde{D}+\beta^2 - 2\tilde{D}. $$
Since both sides of the inequality are positive (recall $\beta > 1$), we may square both sides and get
$$ (\tilde{D}+\beta)^4 - 4 \tilde{D} (\tilde{D}+\beta)^2 \leq (\tilde{D}+\beta)^4 - 4 \tilde{D} (\tilde{D}+\beta)^2 + 4\tilde{D},$$
which is always true (and therefore so is the claim). Thus, we always have $\tilde{G}_+ < 0$,
guaranteeing a maximum principle for $\rho^{1+\gamma_+} \theta$.\\
\\
A similar calculation shows that
$$ \tilde{G}_- = \frac{\kappa_2 f \rho}{1 + \gamma_-} \left( \tilde{D}(1+\gamma_-) + \beta \gamma_- \right). $$
Recall that $\gamma_- < -1$. Then $ \tilde{G}_- > 0 $, 
which gives a minimum principle for $\rho^{1+\gamma_-} \theta$.\\
\\
Putting it all together, we have
$$ \rho^{1+\gamma_+} \theta \leq c_1 \ \ \ \text{ and } \ \ \ \rho^{1+\gamma_-} \theta \geq c_2.$$
The second inequality implies \emph{\textbf{I}} (the positivity of the temperature). Further, both inequalities
imply \emph{\textbf{II}}, as $\rho^{\gamma_+-\gamma_-} =
\rho^{\sqrt{\beta^2 + 2(\beta-2)\tilde{D} + \tilde{D}^2}} \leq c_1/c_2$. So
the density is bounded above unconditionally.\\
\\
Since $\rho$ is bounded above, the two inequalities imply that the density must vanish
if the temperature either blows up or goes to zero. This shows \emph{\textbf{III}}, as well as \eqref{drate}.

\end{proof}

\section{The Maximum/Minimum Principle for Thermal Porous Media Model}
\noindent
The non-isothermal porous media model \eqref{nonisopm} is more complicated than \eqref{idmodel}, leading
to more intricate calculations for the auxiliary variables. Although they ultimately take the same form as before
($\theta \rho f(\rho)$), and there are still exactly two possibilities for $f$, it is no longer possible to find clean
expressions for $f$ even when $\kappa_3$ takes
a simple form. The corresponding maximum and minimum principles also become more conditional, and we
must resort to asymptotic analysis (i.e., large $\rho$ and vanishing $\rho$ limits) to determine precisely which
case occurs for each of the auxiliary variables. The proof is similar to that of Theorem \ref{thm:idmodel-maxmin},
but more technical and involved. For this reason, we present it here at the end, so that the proof of the previous
section can be used as a reference.

\begin{thm}\label{maxpme}
If $(\rho,\theta)$ is a smooth solution pair to \eqref{nonisopm} on $[0,T) \times \mathbb{T}^n$ and
the material conductivity $\kappa_3$ is given by
\begin{equation}
\kappa_3 = a D \theta.
\label{assumption_k_3}
\end{equation}
with constants $D>0$ and $a:=\alpha-1>0$ (so independent of $\rho$), then we have\\

\noindent\textbf{I. High density case:} There is some threshold $\overline{\rho}$ for which, if $\rho(t,x) > \overline{\rho}$
on $\mathbb{T}^n$, then there are constants $c_1, c_2 > 0$ depending on the data such that
$$ \rho^{a+1} \theta \geq c_1 \ \ \ \text{ and } \ \ \ \rho \exp \left(-\frac{\kappa_1 a}{\kappa_2(a+1)} \rho^a \right) \theta \geq c_2,$$
and hence
$$ \theta \geq \max \left( c_1 \rho^{-a-1}, c_2 \rho^{-1} \exp \left(\frac{\kappa_1 a}{\kappa_2(a+1)} \rho^a \right) \right). $$\\
\noindent\textbf{II. Low density case:}
There exists a threshold $\underline{\rho}$ for which, if $\rho(t, x)<\underline{\rho}$ on $\mathbb{T}^n$, then the temperature
$\theta$ is bounded from above. Moreover, there are constants $c_1, c_2 > 0$ depending on the data such that
$$ \theta \leq c_1 \ \ \ \text{ and } \ \ \ \rho\theta \exp \left( \frac{D}{\kappa_2(a+1)} \rho^{-a-1} \right) \geq c_2,$$
and hence
\begin{equation}\label{lbdrho}
\begin{aligned}
\rho \exp \left( \frac{D}{\kappa_2(a+1)}\rho^{-a-1} \right) \geq \frac{c_2}{c_1}.
\end{aligned}
\end{equation}
\end{thm}
\noindent
It is worth noting that, in the case of low density, the estimates hold indefinitely for certain initial data.
The function on the left side of \eqref{lbdrho} is decreasing as $\rho$ increases from zero, so if the constants $c_1$ and $c_2$ are appropriately
chosen (i.e., the initial data is appropriately chosen), then \eqref{lbdrho} would fail if $\rho$ became too large.
If this threshold is less than $\underline{\rho}$, then
the case of very low density also becomes \emph{self-maintaining}; the density and temperature stay bounded above,
and if temperature vanishes somewhere then so must density.
Unfortunately, in the high density case, nothing prevents the temperature from becoming arbitrarily large.
This then allows the density to drop, which means the lower bounds no longer apply. The case of very high density is
\emph{not self-maintaining}.

\begin{proof}
We first eliminate the $\kappa_1$ constant by rescaling. If we define
$$ \tilde \rho := \kappa_1^{\frac{1}{a}} \rho, $$
then \eqref{nonisopm} (where by abuse of notation we still write $\rho$ instead of $\tilde\rho$) becomes
\begin{equation}
\left\lbrace
\begin{split}
&\partial_t \rho = a \Delta (\theta \rho^{a+1})\\
&\kappa_2 \partial_t (\rho \theta) =
\nabla \cdot \left( (a\kappa_2 \theta + a^2 \theta \rho^a ) \nabla (\theta \rho^{a+1}) \right)
+ \nabla \cdot (\kappa_3 \nabla \theta)
\end{split}
\right.
\label{eq:pm2}
\end{equation}
Let $f: (0,\infty) \rightarrow (0,\infty)$ be a smooth monotone function to be determined later.
We look at the quantity $f(\rho) \rho \theta$. Taking a derivative and using \eqref{eq:pm2} yields
\begin{equation}
\begin{split}
\kappa_2 \partial_t (f\rho\theta) &= \kappa_2 \rho \theta f' \partial_t \rho + f \kappa_2 \partial_t (\rho \theta)\\
&= \kappa_2 a \rho \theta f' \Delta (\theta \rho^{a+1}) +
f \nabla \cdot \left( (a \kappa_2 \theta + a^2 \theta \rho^a)\nabla(\theta \rho^{a+1})
+ \kappa_3 \nabla \theta \right).
\end{split}
\label{eq:011}
\end{equation}
Let $x_0 \in \TT^n$ be a point where $f\rho\theta$ achieves a local minimum (respectively maximum) in space.
Assume that $\theta(x_0) > 0$ and $f(\rho(x_0))\rho(x_0) > 0$.
For the rest of the proof, all quantities are implicitly evaluated at $x_0$, though we suppress the notation.
Then the gradient at this point vanishes, so that
\[
(f' \rho + f) \theta \nabla \rho + f\rho \nabla \theta = 0,
\]
or
\begin{equation}
(f' \rho + f) \nabla \rho = - f \rho \frac{\nabla \theta}{\theta}.
\label{eq:crit-grad2}
\end{equation}
In addition, the quantity
\begin{align*}
L &:= \Delta(f\rho\theta) = \theta \left(f'\rho + f \right) \Delta\rho
+  f\rho \Delta\theta
+ 2 \left( f'\rho + f \right) \nabla\rho \cdot \nabla\theta
+ \theta \left( f''\rho + 2 f'\right) |\nabla \rho|^2\\
& \ = \frac{|\nabla \theta|^2}{\theta} \left( \frac{f''\rho+2f'}{(f'\rho+f)^2} f^2 \rho^2 - 2 f\rho \right)
+ \theta (f'\rho+f)\Delta\rho + f\rho \Delta\theta
\end{align*}
is nonnegative (respectively nonpositive). The last equality used \eqref{eq:crit-grad2}, which in general
allows us to compare terms of the form $|\nabla\rho|^2$, $\nabla\rho\cdot\nabla\theta$, and $|\nabla\theta|^2$
to each other (when they are evaluated at $x_0$).\\
\\
Expanding \eqref{eq:011} then yields
\begin{equation}
\begin{split}
&\kappa_2 \partial_t (f\rho\theta) =
	\left( \kappa_2 a f' \rho^{a+2} \theta + (\kappa_2 a + a^2 \rho^a) f\rho^{a+1}\theta + \kappa_3 f \right) \Delta\theta\\
&\quad +\left( \kappa_2 a (a+1) f' \rho^{a+1} \theta^2 + (\kappa_2 a + a^2 \rho^a) (a+1) f \rho^a \theta^2 \right) \Delta \rho \\
&\quad +\left( \kappa_2 a^2 (a+1) f' \rho^a \theta^2 + a^3 (a+1) f \rho^{2a-1} \theta^2
        + (\kappa_2 a + a^2 \rho^a) a (a+1) f \rho^{a-1} \theta^2 \right) |\nabla \rho|^2 \\
&\quad +\left( 2 \kappa_2 a (a+1) f' \rho^{a+1} \theta + 3(\kappa_2 a + a^2\rho^a)(a+1) f\rho^a \theta
        + a^3 f \rho^{2a} \theta + f \partial_\rho \kappa_3 \right) \nabla\theta \cdot \nabla\rho \\
&\quad + \left( (\kappa_2 a + a^2 \rho^a) f \rho^{a+1} + f \partial_\theta \kappa_3 \right) |\nabla\theta|^2.
\end{split}
\label{eq:012}
\end{equation}
The ultimate goal is to evaluate \eqref{eq:012} at $x_0$, the local minimum (respectively maximum), and obtain
\begin{equation}\label{eq:013}
\kappa_2 \partial_t (f \rho \theta) = \tilde{F} L  + \tilde{G}|\nabla \theta|^2
\end{equation}
with $\tilde{F} \geq 0$ and $\tilde{G}$ nonnegative (respectively nonpositive).
This would show that $f\rho\theta$
satisfies a minimum (respectively maximum) principle for such $f$.\\
\\
At $x_0$, we use \eqref{eq:crit-grad2} and \eqref{assumption_k_3}
to turn \eqref{eq:012} into
\begin{equation}
\begin{split}
\kappa_2 \partial_t (f\rho \theta) &= a \theta \left( \kappa_2 f' \rho^{a+2} + (\kappa_2 + a \rho^a)f\rho^{a+1} + Df \right) \Delta \theta \\
&\quad +a(a+1) \rho^a \theta^2 \left( \kappa_2 f' \rho + (\kappa_2 +a\rho^a)f \right) \Delta \rho\\
&\quad +\frac{a^2(a+1)f^2\rho^2}{(f'\rho+f)^2} \left( \kappa_2 f' \rho^a + a f\rho^{2a-1}+(\kappa_2+a\rho^a)f\rho^{a-1} \right) |\nabla\theta|^2 \\
&\quad -\frac{af\rho}{f'\rho+f} \left( 2\kappa_2(a+1)f'\rho^{a+1}+3(\kappa_2+a\rho^a)(a+1)f\rho^a+a^2f\rho^{2a} \right) |\nabla\theta|^2 \\
&\quad +a \left( (\kappa_2 + a\rho^a) f\rho^{a+1} + Df\right) |\nabla\theta|^2.
\end{split}
\label{eq:014}
\end{equation}
In order for \eqref{eq:013} to hold, there must be some $\lambda \geq 0$ (not necessarily constant)
such that
\begin{align*}
&a\theta \left( \kappa_2 f' \rho^{a+2} + (\kappa_2 + a \rho^a)f\rho^{a+1} + Df \right) = \lambda f \rho, \label{DTheta}\tag{A} \\
&a(a+1) \rho^a \theta^2 \left( \kappa_2 f' \rho + (\kappa_2 +a\rho^a)f \right) = \lambda \theta (f'\rho+f) . \label{DRho}\tag{B}
\end{align*}
Writing $\lambda = a \tilde\lambda \rho^{-1} \theta$ crucially eliminates $\theta$ from both equations above. From \eqref{DTheta} we obtain
$$ \tilde{\lambda} = \kappa_2 \rho^{a+2} \frac{f'}{f} + (\kappa_2 + a\rho^a) \rho^{a+1} + D. $$
Plugging this into \eqref{DRho} then yields
$$ \kappa_2 (a+1)\rho^{a+2} f' + (a+1)(\kappa_2 + a\rho^a)\rho^{a+1} f =
\left( \kappa_2 \rho^{a+2}\frac{f'}{f} + (\kappa_2 + a\rho^a) \rho^{a+1} + D \right) (f' \rho + f). $$
This finally simplifies to
\begin{equation}\label{eq:f_implicit}
(f')^2 \left( \kappa_2 \rho^{a+3} \right) + f f' \left( (a\rho^a - \kappa_2 a + \kappa_2)\rho^{a+2} + D\rho \right)
+ f^2 \left( D - a(\kappa_2+a\rho^a) \rho^{a+1} \right) = 0.
\end{equation}
Using the quadratic formula yields two branches of solutions.
After simplification, this becomes
\begin{equation}\label{eq:f_ode_branch}
f' = f \frac{-(a\rho^a + \kappa_2 (1-a))\rho^{a+1} - D \pm \sqrt{\Delta}}{2\kappa_2 \rho^{a+2}} =: f \Psi_\pm(\rho),
\end{equation}
where the discriminant takes the form
$$ \Delta := \left(D + \rho^{a+1}(a\rho^a - \kappa_2 (a+1)) \right)^2 + 4\kappa_2 a (a+1) \rho^{3a+2}. $$
Note that $\Delta$ is strictly positive for all $\rho \geq 0$.
Thus \eqref{eq:f_ode_branch} implies the existence of two solutions $f_+$ and $f_-$ that
both transform \eqref{eq:014} into \eqref{eq:013}. Unfortunately, these ODE's are not explicitly solvable, and
do not yield simple power laws for $f$.\\
\\
We briefly examine now the asymptotic behavior of $f_\pm$.
One can rewrite \eqref{eq:f_ode_branch} for $f_+$ in the following form:
\begin{equation*}
\begin{split}
\frac{f_+'}{f_+} &=
\frac{(D+a\rho^{2a+1} - \kappa_2 (a+1) \rho^{a+1})^2 + 4\kappa_2 a (a+1) \rho^{3a+2} - (D+a\rho^{2a+1}+\kappa_2 (1-a)\rho^{a+1})^2}
{2\kappa_2 \rho^{a+2} (D + a\rho^{2a+1}+\kappa_2 (1-a)\rho^{a+1} + \sqrt{\Delta})}\\
&= \frac{2(a^2 \rho^{2a+1} + \kappa_2 a \rho^{a+1} - D)}{\rho (D + a\rho^{2a+1}+\kappa_2 (1-a)\rho^{a+1} + \sqrt{\Delta})}.
\end{split}
\end{equation*}
Asymptotically, we have
\begin{equation}\label{eq:fp-asymp}
f_+' \approx - \frac 1 \rho f_+ \ \text{ as } \ \rho \rightarrow 0^+ \ \ \text{ and } \ \
f_+' \approx \frac a \rho f_+ \ \text{ as } \ \rho \rightarrow \infty.
\end{equation}
Therefore $f_+$ has a positive singularity at $\rho = 0$ that grows like $\rho^{-1}$, decreases to a minimum value at some critical
$\rho_{1}$ where $a^2 \rho_1^{2a+1} + \kappa_2 a \rho_1^{a+1} = D$, then becomes increasing and grows like $\rho^{a}$.\\
\\
A similar calculation for $f_-$ shows that
\begin{equation}\label{eq:fn-asymp}
f_-' \approx -\frac{D}{\kappa_2 \rho^{a+2}} f_- \ \text{ as } \ \rho \rightarrow 0^+ \ \ \text{ and } \ \
f_-' \approx -\frac{a^2 \rho^{a-1}}{\kappa_2 (a+1)} f_- \ \text{ as } \ \rho \rightarrow \infty.
\end{equation}
Thus, $f_-$ also has a positive singularity at $\rho = 0$ that grows like $\exp \left( \frac{D}{\kappa_2 (a+1)} \rho^{-(a+1)} \right)$,
stays monotone decreasing, and decays exponentially to zero with profile $\exp \left(-\frac{a}{\kappa_2 (a+1)} \rho^a \right)$. The
difference between $f_+$ and $f_-$ comes from expanding the discriminant: $a \rho^{2a+1} + D +\sqrt{\Delta}$ has
simple asymptotics, but $a \rho^{2a+1} + D - \sqrt{\Delta}$ has cancellations at both limits.\\
\\
We now look at the remaining terms of \eqref{eq:014} in light of \eqref{eq:f_ode_branch}. Specifically, \eqref{eq:014}
has now become \eqref{eq:013} with
$$ \tilde{F} = a \theta \left( \kappa_2 \rho^a \frac{f'\rho+f}{f} + a \rho^{2a} + \frac D \rho \right), $$
and
\begin{equation*}
\begin{split}
\tilde{G} =& - a \left( \kappa_2 \rho^a \frac{f'\rho+f}{f} + a \rho^{2a} + \frac D \rho \right) \left( \frac{f''\rho+2f'}{(f'\rho+f)^2}f^2\rho^2-2f\rho\right) \\
&+ \frac{2a^3(a+1) \rho^{2a+1}}{(f'\rho+f)^2}f^3 + \frac{\kappa_2 a (a^2 - 1) \rho^{a+1} - a^2 (4a+3) \rho^{2a+1}}{f'\rho+f} f^2 \\
&+ \left( a^2 \rho^{2a+1} - \kappa_2 a (2a+1) \rho^{a+1} + aD \right) f \\
=& -a \left( \kappa_2\frac{a+1}{2} \rho^{a+1} + \frac{a}{2} \rho^{2a+1} +\frac{D \pm \sqrt{\Delta}}{2} \right)
\left( \frac{\Psi_\pm \rho}{\Psi_\pm\rho+1} + \frac{(\Psi_\pm'\rho + \Psi_\pm) \rho}{(\Psi_\pm\rho+1)^2} - 2 \right) f \\
&+ \frac{8\kappa_2^2 a^3(a+1)\rho^{4a+3}}{(\kappa_2(a+1)\rho^{a+1}-a\rho^{2a+1}-D\pm\sqrt{\Delta})^2}f\\
&+\frac{2\kappa_2^2a(a^2-1)\rho^{2a+2}-2\kappa_2 a^2(4a+3)\rho^{3a+2}}{\kappa_2(a+1)\rho^{a+1}-a\rho^{2a+1}-D\pm\sqrt{\Delta}} f \\
&+ \left( a^2 \rho^{2a+1} - \kappa_2 a (2a+1) \rho^{a+1} + aD \right) f
\end{split}
\end{equation*}
We now let $\rho \rightarrow \infty$ (or $0^+$) for $\tilde{G}$ to obtain
an asymptotic formula for that term in the limit of large (or vanishing) density. Write $\tilde{G}_\pm$ to correspond
to $f_\pm$. We then obtain
\begin{equation*}
\begin{split}
\tilde{G}_+ &\approx a^2 \frac{a+2}{a+1} \rho^{2a+1} f + \frac{8 a^3}{9(a+1)} \rho^{2a+1} f
-\frac{2a^2 (4a+3)}{3(a+1)} \rho^{2a+1} f + a^2 \rho^{2a+1} f \\
&= \frac{a^2 (2a+9)}{9(a+1)} \rho^{2a+1} f > 0 \ \text{ as } \ \rho \rightarrow \infty.
\end{split}
\end{equation*}
However, the vanishing $\rho$ limit is made more singular due to the fact that
$$ \lim_{\rho\rightarrow 0^+}\Psi_+ \rho + 1 \approx  \frac{\kappa_2(a+1)^2}{4D} \rho^{a+1}. $$
From this we obtain
\begin{equation*}
\begin{split}
\tilde{G}_+ &\approx -a \left( \frac{\kappa_2(a+1)}{2} \rho^{a+1} + D\right)
\left( \frac{4D}{\kappa_2 (a+1)^2} \rho^{-a-1} + \frac{8D^2}{\kappa_2^2(a+1)^2} \rho^{-2a-2}-2 \right) f\\
&\quad\quad\quad + \frac{32 D^2a^3}{\kappa_2^2(a+1)^3} \rho^{-1} f + \frac{4Da(a-1)}{a+1} f + aD f\\
&\approx - \frac{8aD^3}{\kappa_2^2 (a+1)^2} \rho^{-2a-2} f < 0 \ \text{ as } \rho \rightarrow 0^+.
\end{split}
\end{equation*}
Thus, for very low values of $\rho$, $\tilde{G}_+$ is negative. This implies that $f_+ \rho \theta$ has a \emph{minimum principle} for minima
that are above a certain threshold $\overline{\rho}$, and a \emph{maximum principle} for maxima that are below a second threshold
$\underline{\rho}$. If $\tilde{G}_+$ has only one zero, then $\overline{\rho}=\underline{\rho}$. So this corresponds to a ``state
change'' in the material between low density and high density.\\
\\
The corresponding calculations for $G_-$ yields
\begin{equation*}
\begin{split}
\tilde{G}_- &\approx -a \kappa_2 \frac{a+1}{2} \rho^{a+1} f + 2a \kappa_2^2(a+1) \rho f + a \kappa_2 (4a+3) \rho^{a+1} f
+a^2 \rho^{2a+1} f\\
&\approx a^2 \rho^{2a+1} f \ \text{ as } \ \rho \rightarrow \infty,
\end{split}
\end{equation*}
and in the low density case
\begin{equation*}
\begin{split}
\tilde{G}_- &\approx -\kappa_2 a (a+1)\rho^{a+1}
\left( \frac{-\frac{D}{\kappa_2} \rho^{-a-1}}{1-\frac{D}{\kappa_2}\rho^{-a-1}}
+ \frac{(a+1)\frac{D}{\kappa_2} \rho^{-a-1}}{\left( -\frac{D}{\kappa_2} \rho^{-a-1} \right)^2} - 2 \right) f \\
&\quad\quad\quad + \frac{2a^3(a+1)\rho^{2a+1}}{\left(-\frac{D}{\kappa_2} \rho^{-a-1}\right)^2} f
+ \frac{\kappa_2 a (a^2-1) \rho^{a+1}}{-\frac{D}{\kappa_2} \rho^{-a-1}} f + aDf \approx aDf \ \text{ as } \ \rho \rightarrow 0^+.
\end{split}
\end{equation*}
This case is of a different type. Owing to various cancellations, the last term in the formula for $\tilde{G}_-$ is the dominant
one for large $\rho$, and the small $\rho$ limit indicates that $\tilde{G}_-$ is always positive.\\
\\
\emph{Proof of \textbf{I}}: There is some threshold $\overline{\rho}$ for which, if $\rho(t,x) > \overline{\rho}$ on $\mathbb{T}^n$, then $\tilde{G}_\pm$
is strictly positive. Thus we have a minimum principle for the quantities $f_\pm \rho \theta$. Furthermore, as long as this $\overline{\rho}$ is taken
large enough, we may replace $f_\pm$ by their asymptotic profiles, so that we obtain two explicit minimum principles:
$$ \rho^{a+1} \theta \geq c_1 \ \ \ \text{ and } \ \ \ \rho \exp \left(-\frac{a}{\kappa_2 (a+1)} \rho^a \right) \theta \geq c_2. $$
This implies an absolute lower bound on the temperature (if $\theta$ gets too small, $\rho$ has to increase to keep the first quantity above its
minimum, but that makes the second quantity decrease). Rigorously,
$$ \theta \geq \max \left( c_1 \rho^{-1-a}, c_2 \rho^{-1} \exp \left( \frac{a}{\kappa_2 (a+1)} \rho^a \right) \right). $$
The right hand side has a positive minimum, which occurs when $c_1 \rho^{-a} = c_2 \exp (\rho^a/\kappa_2)$, a transcendental expression.\\
\\
\emph{Proof of \textbf{II}}: Similarly, there is a second threshold $\underline{\rho}$ for which, if $\rho(t,x) < \underline{\rho}$ on $\mathbb{T}^n$, then
$\tilde{G}_+ < 0$, $\tilde{G}_->0$, and $f_+ \approx \rho^{-1}$ while $f_- \approx \exp( D\rho^{-a-1} / (\kappa_2 (a+1)))$. This implies
$$ \theta \leq c_1 \ \ \ \text{ and } \ \ \ \rho\theta \exp \left(\frac{D}{\kappa_2(a+1)}\rho^{-a-1} \right) \geq c_2. $$
We get an upper bound for the temperature immediately, and \eqref{lbdrho} follows.

\end{proof}

\section{Acknowledgments}
\noindent
N.-A. Lai was partially supported by NSF of Zhejiang Province(LY18A010008) and NSFC(11771359), C. Liu was partially supported by NSF DMS-1759535 and DMS-1950868 and the United States –Israel Binational Science Foundation (BSF) \#2024246.
A. Tarfulea was partially supported by NSF DMS-2012333.


\end{document}